\documentclass[english,12pt]{elsarticle}
\usepackage[T1]{fontenc}
\usepackage[latin9]{inputenc}
\usepackage[a4paper]{geometry}
\usepackage{float}
\usepackage{textcomp}
\usepackage{amsthm}
\usepackage{amsmath}
\usepackage{amssymb}
\usepackage{graphicx}
\usepackage{algorithm}
\usepackage{algorithmic}
\usepackage{multirow}
\usepackage{booktabs}
\usepackage{rotating}

\newtheorem{theorem}{\bf Theorem}[section]

\makeatletter

\floatstyle{ruled}
\newfloat{algorithm}{tbp}{loa}
\providecommand{\algorithmname}{Algorithm}
\floatname{algorithm}{\protect\algorithmname}

\numberwithin{equation}{section}
\numberwithin{figure}{section}
\theoremstyle{plain}

\usepackage{babel}

\makeatother

\usepackage{babel}
\journal{Expert Systems With Applications}
\biboptions{authoryear}

\providecommand{\theoremname}{Theorem}
\begin{document}
%
%
\title{Hybrid evolutionary algorithm with extreme machine learning fitness function evaluation for two-stage capacitated facility location problem}

\author{Peng
Guo}
 \address{Peng Guo: School of Mechanical Engineering, Southwest Jiaotong University,
 Chengdu,  610031 China.   Email: pengguo318@gmail.com.}
\author{Wenming Cheng}
 \address{Wenming Cheng: School of Mechanical Engineering, Southwest Jiaotong University,
 Chengdu,  610031 China.   Email: wmcheng@swjtu.edu.cn.}

 \author{Yi
 Wang}
 \address{Yi Wang:  corresponding author. Department of Mathematics and Computer Science, Auburn University at Montgomery,
 P.O. Box 244023, Montgomery, AL 36124-4023
 USA.   Email: ywang2@aum.edu.}
\date{}

\begin{abstract}
This paper considers the two-stage capacitated facility location problem (TSCFLP) in which products manufactured in  plants are delivered to customers via storage depots. Customer demands are satisfied subject to limited plant   production and limited depot storage capacity.  The objective is to determine the locations of plants and depots   in order to minimize the total cost  including the fixed cost  and transportation cost. A hybrid evolutionary algorithm (HEA) with genetic operations and local search is proposed. To avoid the expensive   calculation of fitness of population   in terms of computational time, the HEA uses extreme machine learning to approximate the fitness of most of the individuals.  Moreover, two heuristics based on the characteristic of the problem is incorporated  to generate a good initial population.

Computational experiments are performed on two sets of  test instances from the recent literature.
The performance of the proposed algorithm is evaluated and analyzed.
Compared with the state-of-the-art genetic algorithm, the proposed algorithm can find  the optimal or near-optimal solutions in a reasonable computational time.
\end{abstract}

\maketitle

\section{Introduction}
As a   strategic issue in the design of  supply chain networks, logistics facility location problems have been studied extensively in the past few decades, for example, see \citet*{Mesa1996_review,Owen1998_review,Snyder2006_review,Sahin2007_review,Melo2009401_Review}. In a facility location problem, management decision makers have to decide which sites should be chosen to establish new  facilities from a set of available candidate sites, while    constraints are met in order to  minimize the total cost. The constraints are such that the demands of all customers have to be met, the capacity limits of the suppliers and facilities must not be violated, etc. The cost includes fixed costs to open plants and   depots, and variable costs associated with  transportation.
The decision about the facility location leads to long term commitments due to the megacephalic fixed costs of these facilities. The selection of facility location will  profoundly influence the management planning of   organizations, especially in the relation between the sectors and theirs customers.

Owing to the importance of  the facility location problem, it has been widely considered in the literature. In particular, single-stage capacitated facility location problems (CFLP) have been successfully analyzed \citep*{Klose2005LBCPLP,Zhang2005MLSCFLP}. Subsequently, many effective search algorithms have been designed for their near-optimal and optimal solutions of large-sized instances, including heuristics based on Lagrangian relaxation \citep*{Avella2008HECFLP}, kernel search \citep*{Guastaroba2012KSCFLP,Guastaroba2014438SSCFLP} and cut-and-solve algorithm \citep*{Yang2012521LSCFLP}. Moreover, meta-heuristics, such as tabu search \citep*{Sun2011TSCFLP}, simulated annealing,  genetic algorithm \citep*{Arostegui2006empirical_CFLP} and firefly algorithm \citep*{Rahmani201470CFLP_HFGA} were developed to solve the CFLP.

Recently, two-stage capacitated facility location problem (TSCFLP) has attracted researchers' attention, as in many situations   more than one type of facilities are  considered simultaneously.
Due to the intractability of the TSCFLP, it is hard to obtain optimal solutions for   a large-sized instance. Therefore, some researchers focused on heuristics rather than mathematical modeling in solving the TSCFLP.
\citet*{Klose1999LPH} proposed a linear programming based heuristic for TSCFLP with single source constraints, and evaluated the proposed algorithm on a large set of test problems with up to 10 plants, 50 potential depots and 500 customers. They   presented a Lagrangian relax-and-cut procedure for the problem, and found  the proposed algorithm could yield very good solutions in the expense of  consuming a lot of effort to optimize or re-optimize the Lagrangian dual \citep*{Klose2000TSCFLP}.

In order to solve  large sized instances of the TSCFLP, \citet*{Fernandes2014200}  suggested a simple genetic algorithm (GA) to determine which plants and depots should be opened, and obtained the flow values between plants and depots and between depots and customers by solving a minimum cost flow problem.
In addition, the two-stage uncapacitated facility location problem was considered by \citet*{Marin20071126}, and its lower bounds were delivered by a relaxation formulation.
\citet*{Wang2012TSCFLP} studied a TSCFLP with fuzzy costs and demands, and developed a particle swarm optimization to solve the problem under consideration.

Although several heuristics have been developed to solve the TSCFLP, most methods are not fit for solving  large-sized problems in terms of solution quality and computational time.
For examples, the computational time of GA  varies from 236 seconds to 2784 seconds for the problem instances with 50 plants and 100 plants in \citet*{Fernandes2014200}, and
each function evaluation  involves solving a minimum cost flow problem which is computationally expensive. Specifically, the GA requires a very large number of fitness function evaluations for producing a near-optimal solution.

For most evolutionary algorithms (EAs) like the GA, a very large number of  fitness function evaluations are required before a well acceptable solution can be delivered.
This characteristic of the EAs seriously limits their applications in solving intractable high-dimensional and multimodal optimization problems.
One promising way to significantly ease the computational burden of the EAs is to adopt computationally cheap surrogate models instead of computationally expensive fitness evaluations \citep*{Jin2005FA}.
Various techniques for the construction of surrogate models (also called approximation models or meta-models) have been employed to obtain the efficient and effective hybrid EAs. Among these techniques, artificial neural network (ANN), support vector machine (SVM) and kriging models are among some of the most prominent and commonly used techniques \citep*{Hacioglu2007FEA_NN,zhang2010expensive,dias2014genetic,Zheng2015GA_SVM}.
By elaborating surrogate models, the computational burden can be greatly reduced. This is due to  the efforts  in constructing the surrogate models and then using it to predict fitness values  are much lower than the efforts of   directly calculating the fitness functions by the standard approach.
Inspired by the application of surrogate models in solving complex continuous optimization problems, it is intended in this paper to adapt a meta-heuristic algorithm with fitness approximation to solve the TSCFLP, in order to    reduce the computational time and guarantee correct convergence.

The Extreme learning machine (ELM) developed by \citet*{Huang2006489ELMO} strikes a balance between speed and generalization performance, and attracts more and more attention from various respects.
Compared with the ANN, the SVM and other traditional forecasting models, the ELM model retains the advantages of fast learning, good ability to generalize and convenience in terms of modeling.
In this paper,  a hybrid evolutionary algorithm with fitness approximation (HEA/FA) is proposed to obtain the optimal or neat-optimal solutions for the TSCFLP. In the frame of the algorithm, the extreme machine learning is used as surrogate model to approximate   fitness values of most individuals in the population.
The contribution of this paper is to introduce the basic structure of the evolutionary algorithm with the integration of the ELM fitness approximation when solving a discrete optimization problem.
The proposed algorithm uses the genetic operations (selection, crossover and mutation) as well as restarting strategy and a special local search operation with the best individual to update the current population. Computational results and comparison to state-of-the-art GA proposed in the literature demonstrate the effectiveness and efficiency of the proposed HEA/FA approach.

The remainder of this paper is organized as follows. In Section \ref{sec:SecMIP}, the TSCFLP is described in details and formulated as a mixed integer programming model. In Section \ref{sel:ELM}, the ELM algorithm is briefly introduced for the convenience of readers. Section \ref{sec:secHMA} presents the hybrid evolutionary framework of optimizing the TSCFLP using fitness function approximation via the ELM. Subsequently, the numerical tests and comparisons are carried out in Section \ref{sec:secNR}. Finally Section \ref{sec:secCon} summarizes some conclusions and points out future research.
\section{Problem  formulation\label{sec:SecMIP}}
The problem under study is defined as follows: a single product is produced at plants and then transported to depots, while both plants and depots  have limited capacities. From the depots the product is delivered to customers to satisfy their demands. The use of  plants/depots is accompanied by a fixed cost, while transportation from the plants to the customers via the depots results in a variable cost. The two-stage facility location problem aims to identify what plants and depots to use, as well as the product flows from the plants to the depots and then to the customers, such that the demands are met and the total cost is minimized. The problem under consideration in this paper is the same as that in the reference \citep*{Litvinchev2012}.

The parameters of this problem under study are defined as follows. Let $I, J, K$ be the sets of plants, depots and customers, respectively. Let
$f_i$ be the fixed cost  for plant $i\in I$;
$b_i$ be the capacity of plant $i \in I$;
$g_j$ be the fixed cost associated with depot $j \in J$;
$p_j$ be the capacity of depot $j \in J$;
$c_{ij}$ be the cost of transporting one unit of the product from plant $i \in I$ to depot $j \in J$;
$d_{jk}$ be the cost of transporting one unit of the product from depot $j\in J$ to customer $k \in K$;
$q_k$ be the demand of customer $k \in K$.

Let binary variables $y_i$, $i\in I$, be equal to 1 if and only if plant $i$ is chosen to be opened.
 Similarly, let $z_j $ be a binary variable which is equal to 1 if and only if depot $j\in J$ is opened, otherwise it is equal to 0.
 And let real variables $x_{ij}$ and $s_{jk}$ define the flow of products from plants to customers via depots. Using the above notations, the TSLFLP can be formulated as the following integer programming model.
\begin{eqnarray}
\mathrm{Minimize}\qquad \qquad  Z=\sum_{i \in I}f_iy_i+\sum_{j\in J}g_jz_j+\sum_{i\in I}\sum_{j\in J}c_{ij}x_{ij}+\sum_{j\in J}\sum_{k\in K}d_{jk}s_{jk}  \label{eq:objective}
\end{eqnarray}
subject to
\begin{align}
\sum_{j\in J}s_{jk}  &\geq   q_k  & \forall \ k  &\in K  \label{eqn:eq2}\\
\sum_{i\in I}x_{ij}  &\geq   \sum_{k \in K}s_{jk}   &  \forall \ j   &\in J \label{eqn:eq3}\\
\sum_{j\in J}x_{ij} &\leq b_iy_i & \forall \ i  &\in I \label{eqn:eq4}\\
\sum_{k \in K}s_{jk} &\leq p_jz_j & \forall \ j & \in J \label{eqn:eq5}\\
x_{ij}& \leq b_iz_j & \forall i \in I,& \ j \in J \label{eqn:aeq2}\\
x_{ij},\ s_{jk}\geq 0, &y_i, z_j \in \{0,1\}, &\forall \ i \in I, \ j &\in J,\ k\in K \label{eqn:eq6}
\end{align}

The objective function \eqref{eq:objective} includes the opening costs of the plants and the depots and the transportation costs in the two-stages. Constraints \eqref{eqn:eq2} ensure the demand from each customer must be met. Constraints \eqref{eqn:eq3} guarantee that the total amount of products transported from a depot must be at most the total transported to it from the plants. Constraints \eqref{eqn:eq4} and \eqref{eqn:eq5} gives the capacity limits for plants and depots and assure the flow only from plants and depots opened.
Constraints \eqref{eqn:aeq2} determine the upper bounds of  the decision variables $x_{ij}$.
At last, constraints \eqref{eqn:eq6} represent the boundary values of the decision variables.

The model considered in this paper is a mixed integer linear programming (MILP) model and is characterized by binary variables $y_i$, $z_j$ , and continuous variables $x_{ij}$, $s_{jk}$. Let $y=(y_i:i\in I)$, $z=(z_j:j\in J)$, $x=(x_{ij}:i\in I, j\in J)$ and $s=(s_{jk}:j\in J, k\in K)$. A solution to the TSCFLP problem shall be denoted by $\mathcal{X}=\{y,z,x,s\}$.

The capacitated facility location problem (CFLP) is equivalent to the second stage of the TSCFLP, so it  can be looked as a special case of the TSCFLP. Since the CFLP was proved to be NP-hard by \citet*{Cornuejols1990uncapacitated}, the TSCFLP under consideration is also NP-hard.
 Since polynomial time algorithms are both theoretically and practically impossible for the large-sized TSCFLP,  it is necessary to develop heuristic or meta-heuristic approaches.
In order to accelerate the solution speed, a hybrid
evolutionary algorithm with fitness approximation (HEA/FA) based on the extreme learning machine is developed to obtain near-optimal solutions in this paper.

\section{Extreme learning machine}\label{sel:ELM}
The Extreme learning machine (ELM), which was proposed by \citet*{Huang2006489ELMO}, is a relatively new learning algorithm of single hidden-layer feed-forward neural networks (SLFNs). It  selects hidden nodes at random and analytically determines the output weights of the SLFNs. Unlike conventional  gradient-based learning methods, the ELM does not need to tune the parameters between the input layer and the hidden layer iteratively. All hidden node parameters are assigned randomly, which are independent of the training data. Different from the  traditional learning algorithms for neural networks the ELM not only attempts to reach the smallest training error but also the smallest norm of output weights.

Let $n$, $h$ and $m$ denote the node numbers of the input layer, the hidden layer and the output layer, respectively. Let $R$ denote the set of real numbers and $R^n$ denote the $n$-dimensional real space. Let $A^T$ be the transpose of a matrix $A$.  Suppose there are $N$ arbitrary distinct training samples $(\textbf{x}_j,\textbf{t}_j)\in R^n \times R^m$, where $\textbf{x}_j$ is an   input vector in $R^n$ and $\textbf{t}_j$ is a target vector in $R^m$, and $1\le j\le N$.   The SLFNs with $h$ hidden nodes
 and activation function
$g(\cdot)$ are mathematically modelled as:
\begin{equation}\label{eq:eq_ELM1}
  \sum_{i=1}^{h}\beta_ig(\textbf{w}_i^T\textbf{x}_j+b_i) + \beta_{0}=\textbf{o}_j, \quad j=1,\ldots, N
\end{equation}
where $\textbf{w}_i=[\omega_{i1}, \omega_{i2},\ldots,\omega_{in}]^T$, $1\le i \le h$, is the weight vector connecting the $i$th hidden node and the input node $\textbf{x}_j$,  $\beta_i=[\beta_{i1},\beta_{i2},\ldots,\beta_{im}]^T$ is the weight vector connecting the $i$th hidden node and the output nodes,   $b_i\in R$ is the threshold of the $i$th hidden node, and $\beta_0$ is a bias parameter vector in $R^m$. Moreover,
 $\textbf{o}_j=[o_{j1},o_{j2},\ldots,o_{jm}]^T$, $1\le j \le N$, is the output vector of the SLFNs.

That the SLFNs with $h$ hidden nodes with activation function $g(\cdot)$ can approximate the $N$ samples with zero error means that $\sum_{j=1}^N\parallel \mathbf{o}_j-\mathbf{t}_j\parallel=0$ for a chosen norm, i.e., there exist $\beta_i, \beta_0 \in R^m$,$ \textbf{w}_i \in R^n$, $b_i\in R$,     such that
\begin{equation}\label{eq:eq_ELM2}
  \sum_{i=1}^{h}\beta_ig(\textbf{w}_i^T \textbf{x}_j+b_i)+\beta_0=\textbf{t}_j, \quad j=1,\ldots, N.
\end{equation}
The above $N$ equations can be rewritten compactly as
\begin{equation}\label{eq:eq_ELM3}
  \textbf{H}\beta=\textbf{T}
\end{equation}
where $\textbf{H}\in R^{N\times (h+1)}$ is the hidden-layer output matrix of the ELM
\begin{equation}\label{eq:eq_ELM4}
\begin{split}
  \textbf{H}(\textbf{w}_1,\ldots,\textbf{w}_h,b_1,\ldots,b_h,\textbf{x}_1,\ldots,\textbf{x}_N)\\
  =\begin{bmatrix}
   1&   g(\textbf{w}_1\cdot\textbf{x}_1+b_1) & \cdots & g(\textbf{w}_h\cdot\textbf{x}_1+b_h) \\
     \vdots & \cdots & \vdots \\
   1&  g(\textbf{w}_1\cdot\textbf{x}_N+b_1) & \cdots & g(\textbf{w}_h\cdot\textbf{x}_N+b_h) \\
   \end{bmatrix}_{N\times (h+1)}
   \end{split}
\end{equation}
$\beta=[\beta_0, \beta_1,\beta_2,\ldots,\beta_h]^T \in R^{(h+1)\times m}$ denotes the {\em matrix} of output weights, and $\textbf{T}=[\textbf{t}_1,\textbf{t}_2,\ldots,\textbf{t}_N]^T\in R^{N\times m}$ denote the target matrix of training data.
As described in \citet*{Huang2006489ELMO},  the $i$th column of $\textbf{H}$ is the $i$th hidden node output with respect to inputs $\textbf{x}_1,\textbf{x}_2,\ldots,\textbf{x}_N$. When the randomly selected input weights $\textbf{w}_i$ and the hidden layer thresholds $b_i,  $ $1\le i \le h$, are given, the network training is to  find  a least squares solution $\beta$ of the linear system. The least norm least-squares solution of the above linear system is
\begin{equation}\label{eq:eq_ELM5}
  {\beta}^*=\textbf{H}^{\dag}\textbf{T}
\end{equation}
where $\textbf{H}^{\dag}$is the Moore-Penrose pseudo inverse of the hidden layer output matrix $\textbf{H}$. Thus, the procedure of ELM can be described in three main steps as follows:

Algorithm ELM: Given a training set   ${S}=\{(\textbf{x}_i,\textbf{t}_i)|\textbf{x}_i\in\mathrm{\textbf{R}}^n,
\textbf{t}_i\in\mathrm{\textbf{R}}^m,i=1,\ldots,N \}$, activation function $g(\cdot)$, and  number of hidden nodes $h$,
\begin{enumerate}
  \item Generate the weight vector $\textbf{w}_i$ and bias $b_i$, $i=1,\ldots,h$ at random. {In particular, the weight vector $\textbf{w}_i$ is drawn from the standard uniform distribution on the open interval (-1,1), and the bias $b_i$ is picked from the uniform distribution over the interval (0,1).}
  \item Calculate the hidden layer output matrix $\textbf{H}$.
  \item Estimate the output weight $\beta^*=\textbf{H}^{\dag}\textbf{T}$.
\end{enumerate}

Since the ELM was developed, a number of researchers have already applied the ELM to many application fields, such as Medical/Biomedical, computer vision, image/video processing, chemical process, fault detection and diagnosis, and so on \citep*{Huang201532REVIEW}.
Compared with other state-of-the-art learning algorithms, including support vector machine (SVM) and deep learning, the ELM algorithm is fast and stable in training, easy in implementation, and accurate in modeling and prediction \citep*{Huang2011ELMSURVEY}.

\section{Hybrid evolutionary algorithm with fitness approximation for TSCFLP}\label{sec:secHMA}
In this section, we propose a   hybrid evolutionary algorithm (HEA) specialized for the TSCFLP   in detail. Besides incorporating successful elements of the previously   constructed effective heuristic algorithms \citep*{Fernandes2014200}, the proposed HEA gains pretty good balance between quality of solutions and computational time in two respects:
1) incorporating local search strategy to enhance the quality of the elitist individuals, and 2) introducing approximating fitness evaluation approach based on the ELM to alleviate the computational burden of the HEA. The idea of introducing fitness approximation has been proven to be effective and efficient in solving complex continuous optimization problems \citep*{Jin2002framework,Hacioglu2007FEA_NN,Zheng2015GA_SVM}.
\subsection{Framework of the proposed algorithm}
In this paper, the generic algorithm (GA) is employed to work as the search framework of the evolutionary computation algorithm owing to its success application to solving  the CFLP \citep*{Jaramillo2002LP_GA,Correa2004CPP_GA,Arostegui2006empirical_CFLP}. The  GA was inspired by biological mechanisms such as the genetic inheritance laws of Mendel and the natural selection concept of Charles Darwin, where the best adapted species survive \citep*{Holland1973genetic}. In order to use a GA to solve a specific problem, a suitable representation (or encoding) is firstly considered. The candidate solution of the problem are encoded by a set of parameters and referred to as chromosomes or individuals.
The word `individual' and the word `solution' are used interchangeably herein.
A fitness value is directly calculated by the objective function of the problem under consideration for each individual.
At each iteration, the GA evolves the current population of individuals  into a new population through selection, crossover and mutation processes. Only the fittest individuals are selected from the  current population and retained in next iteration \citep*{Gen2000genetic}. Regarding the stopping criterion, two commonly used termination conditions are used in the proposed algorithm herein. If the number of iterations $t$ exceeds the maximum number of iterations $t_{\mathrm{max}}$ or the best solution is not improved upon in $t_{\mathrm{nip}}$ successive iterations, the algorithm is terminated; otherwise, a new iteration is started and the iteration timer $t$ is increased by 1.

The proposed algorithm can be separated into six main parts: (1) initialization operation, (2) selection operation, (3) crossover operation, (4) mutation operation, (5) fitness approximation operation, and  (6) local search operation.
The procedure of the HEA with fitness approximation (HEA/FA) is shown in Algorithm \ref{alg:alg_Framework}.
The details of the proposed algorithm presented for the TSCFLP are elaborated as follows.

\begin{algorithm}[htp]
\begin{algorithmic}[1]\small

\caption{\label{alg:alg_Framework} The hybrid evolutionary algorithm with fitness approximation}
\STATE Set the value of the population size $N_p$, the maximum number of iterations $t_{\mathrm{max}}$, the maximum number of consecutive non-improvement iterations $t^*_{\mathrm{nip}}$;
\STATE Generate a population that consists of 2 individuals and $2N_p-2$ randomly generated individuals. The first two individuals are  obtained from the heuristics CBR (see Algorithm 2) and a rounding heuristics. Further,  the algorithm MIH (see Algorithm 3) is applied to correct and improve each individual to    guarantee its feasibility.
\STATE Evaluate all  $2N_p$ individuals using the exact fitness function.
Use these $2N_p$ individuals and their fitness values to form a training set $S$. Train the ELM model using $S$.
Select the first $N_p$ individuals that have smaller fitness values from the $2N_p$ individuals  to form the initial population $P$;
\STATE Set $t=1$ and $t_{\mathrm{nip}}=0$;
\WHILE {the stopping criterion is not met }
\FOR {$i:=1$ to $N_p$}
\STATE Select two individuals $X_{\imath}$, $X_{\jmath}$ from  population  $P$ randomly;  /\% {\em Selection Operation \uppercase\expandafter{\romannumeral1} } \%/
\STATE Apply Crossover operator to the selected solutions $X_{\imath}$ and $X_{\jmath}$ to obtain one offspring solution based on the adaptive crossover probability;  /\% {\em Crossover Operation} \%/
\STATE Add the offspring individuals to a candidate population $P'$ excluding any identical candidate;
\ENDFOR
\FOR {$i:=1$ to $|P'|$}
\STATE Apply mutation operator to the candidate population $P'$ based on the adaptive mutation probability, excluding any identical candidate; /\% {\em Mutation Operation} \%/
\ENDFOR
\STATE Estimate all individuals of the candidate population $P'$ with the ELM-based approximation model and find the best individual $X'_{\mathrm{best}}$ among these candidate individuals;
\STATE Generate a set $Q^*$ of all possible neighboring individuals relative to the best individual $X'_{\mathrm{best}}$  using {\em Local Search} strategy;
\STATE Estimate the individuals of the set $Q^*$ with the ELM fitness approximation model, and select the best individual from the set $Q^*$ to update the best individual $X'_{\mathrm{best}}$ based on their approximated fitness values;
\STATE Calculate the  best $N_e$ individuals of the population $P'$ with exact fitness function, where, $N_e\ll N_p$, for example, $N_e=10\%N_p$. Moreover, these $N_e$ individuals with exact fitness values are added into the training set to update the ELM model.
\STATE Choose the best individual $X^*$ from these individuals to update the   best individual $X_{\mathrm{best}}$ found so far if $Z(X^*)<Z(X_{\mathrm{best}})$, and set the parameter $t_{\mathrm{nip}}=0$; otherwise, do not update $X_{\mathrm{best}}$ and increase $t_{\mathrm{nip}}$ by 1.
\STATE Combine the populations $P$ and $P'$ to form a population set $\tilde{P}$, and select the best $N_p$ individuals from the set $\tilde{P}$ based on their fitness values to update population $P$; /\% {\em Selection Operation \uppercase\expandafter{\romannumeral2} }\% / 
\STATE Perform the restarting strategy on the population $P$;
\STATE $t=t+1$;
\ENDWHILE
\STATE Output the   best individual $X_{\mathrm{best}}$ found by the algorithm and its objective function value.
\end{algorithmic}
\end{algorithm}

\subsection {Encoding and  population initialization}
 The encoding of individuals used in our implementation is the same as that of many studies \citep*{Fernandes2014200,Rahmani201470CFLP_HFGA}.  Once the binary decision variables $y=(y_i:i\in I)$ and $z=(z_j:j\in J)$ of  the MILP model proposed in Section \ref{sec:SecMIP}  are determined,
the MILP model is simplified into a  linear programming (LP) model with the flow variables $x=(x_{ij}:i\in I,j\in J)$ and $s=(s_{jk}:j\in J,k\in K)$.  The simplified LP model can be easily handled by various commercial linear programming softwares. In this paper, Gurobi v6.5 is adopted to obtain the solution of the LP model optimally.
The real variable $x_{ij}$ is limited in the interval $[0, b_i], i=1,\ldots,I$ for each $j\in J$, and the values   of the  variable $s_{jk}$ lie in the interval $[0, p_j], j=1,\ldots,J$ for each $k\in K$.
The LP model is thus given by \eqref{eq:objective}-\eqref{eqn:eq5}.
In the population of individuals of the  HEA/FA algorithm, each individual contains  $|I|+|J|$ binary elements, which refer to all plants and depots in the TSCFLP.  The binary variables for one individual shall be collectively denoted by $X=(y,z)$. The values of these elements indicate whether corresponding facility would be selected. The value of 1 shows the corresponding plant (or depot) is selected to  open, otherwise it is closed. 

 Generally, the initial population of individuals is filled with $N_p$  randomly generated binary vectors. In order to guarantee an initial population with certain quality and diversity, two heuristics  are used to produce two individuals with special property.  The remaining individuals are initialized with random binary vectors.

For the first heuristic, a cost-benefit  index constructed by \citep*{Fernandes2014200} is used to decide which facility to open.
Both fixed and transportation costs are integrated to form the priority index for each plant or depot.
The  index is defined as the ratio of  the sum of  a facility's fixed cost and  its related transportation costs to its capacity.
For  plant $i\ (i \in I)$, its cost-benefit  index is stated by $({f_i+\sum_{j \in J}c_{ij}})/{b_i}$.
Once the   plants to open are determined,  the index of a depot $j\ (j \in J)$ is delivered by $({\sum_{i\in I}c_{ij}+g_j+\sum_{k\in K}d_{jk}})/{p_j}$.

In this paper, facilities with   smaller index are selected sequentially.
 Based on the above description, the heuristic is called cost-benefit ranking (CBR) algorithm, and its procedure is described in Algorithm \ref{alg:alg_CBR}.
\begin{algorithm}[htp]
\begin{algorithmic}[1]\small
\caption{\label{alg:alg_CBR} The CBR algorithm}
\STATE Initialize: $y_i:=0$, $\forall i \in I$, $z_j:=0$, $\forall j \in J$;
\STATE Calculate the cost-benefit index for each plant   using the expression $({f_i+\sum_{j \in J}c_{ij}})/{b_i}$;
\WHILE {$\sum_{i \in I}b_iy_i \leq \sum_{k \in K}q_k$ }

\STATE Select one plant $\imath$ with the smallest index from  the set of unconsidered plants;
\STATE $y_{\imath}=1$;
\STATE Delete  plant $\imath$ from the set of the unconsidered plants;
\ENDWHILE
\STATE Calculate the cost-benefit index for each depot   using the expression $({\sum_{i\in I}c_{ij}+g_j+\sum_{k\in K}d_{jk}})/{p_j}$;
\WHILE {$\sum_{j \in J}p_jz_j \leq \sum_{k \in K}q_k$}
\STATE Select one depot $\jmath$ with the smallest index from  the set of unconsidered depots;
\STATE $z_{\jmath}=1$;
\STATE Delete  depot $\jmath$ from the set of the unconsidered depots;
\ENDWHILE
\STATE Combine $y=(y_i:i\in I)$ and $z=(z_j:j\in J)$ into one complete individual.
\end{algorithmic}
\end{algorithm}

In the second heuristic, a rounding method for generating an initial individual of the population is presented. Firstly,
the optimal solution of the linear programming model \eqref{eq:objective}-\eqref{eqn:eq6} with  continuous variables $y$ and $z$ in the interval [0, 1] can be obtained by using Gurobi easily.
Then the rounding operation is performed on the binary variables $y_i$, $i\in I$ and $z_j$, $j\in J$ in the optimal solution.
If $y_i$ (or $z_j$)    is less than 0.5, it is set to 0; otherwise, it is set to 1. Once the rounding operation is done for each element in the solution, one individual with binary values can be obtained. The individual is regarded as the output of the heuristic.

However, the  individuals delivered by the above approaches  may be infeasible owing to the capacity constraints of plants and depots, i.e. $\sum_{i\in I}b_iy_i\geq \sum_{k\in K}q_k$ and $\sum_{j\in J}p_jz_j\geq\sum_{k\in K}q_k$, may not be met.
In our heuristic algorithm, an approach with the cost-benefit index is designed to correct the infeasible individuals obtained from the rounding heuristic.
%
The approach
is based on the  observation  in Theorem \ref{Lemma_sum}.

\begin{theorem}\label{Lemma_sum}
In any optimal solution for the TSCFLP, the following equality
\begin{equation}\label{eq:eq_lemma_sum}
  \sum_{i \in I}\sum_{j \in J}x_{ij}= \sum_{j\in J}\sum_{k \in K}s_{jk}=\sum_{k \in K}q_k
\end{equation}
must be met.
\end{theorem}

\begin{proof}
The second equality in \eqref{eq:eq_lemma_sum}  is considered first. Recalling the constraints \eqref{eqn:eq2}, an optimal solution must meet the following equality
$$
\sum_{j\in J}s_{jk}=q_k+\mu_k,\qquad k\in K,
$$
where $\mu_k\ge 0$, $k\in K$ are   slack variables. From the above equation and noting $q_k$, $k\in K$ are constants, it is clear that only if each $\mu_k=0$, $k\in K$, the cost term $\sum_{j\in J}\sum_{k\in K}d_{jk}s_{jk}$ appearing in the objective function \eqref{eq:objective} is minimal. Thus by setting $\mu_k=0$, $k\in K$, the second equality $\sum_{j\in J}\sum_{k \in K}s_{jk}=\sum_{k \in K}q_k$ is obtained.

In a very similar spirit, the first equality is proved. By constraints \eqref{eqn:eq3} and the above argument, an optimal solution must meet the following
$$
\sum_{i \in I}\sum_{j \in J}x_{ij}\ge  \sum_{j\in J}\sum_{k \in K}s_{jk}=\sum_{k \in K}q_k.
$$
Constraints \eqref{eqn:eq3} also implies that
$$
\sum_{i\in I}x_{ij}=\sum_{k\in K}s_{jk} +\tilde{\mu}_j, \qquad j\in J,
$$
where $\tilde{\mu}_j\ge 0, \ j\in J$ are slack variables. This makes it clear that only when each $\tilde{\mu}_j= 0, \ j\in J$, can the cost term $\sum_{i\in I}\sum_{j\in J}c_{ij}x_{ij}$ be minimal. Of course, in turn, this implies the first equality in equation \eqref{eq:eq_lemma_sum}.  The theorem has been proved.

\end{proof}

In our proposed algorithm,   infeasible individuals are   corrected based on the cost-benefit indexes of   facilities. Afterward, the corrected individuals are improved based on Theorem \ref{Lemma_sum}. The entire correction procedure is named {\em modified and improved heuristic} (MIH). Its detailed steps are showed in Algorithm \ref{alg:alg_MH}.

\begin{algorithm}[htp]
\begin{algorithmic}[1]\small
\caption{\label{alg:alg_MH} The MIH algorithm}
\STATE \textbf{Input}: A given individual $X$, the parameters of the TSCFLP, $y_i:=X_i$, $\forall i \in I$, $z_j:=X_{|I|+j}$, $\forall j \in J$,
the corresponding improved feasible individual $X^{c}\leftarrow$ [ ]
, and its two components $y_i^c\leftarrow y_i$, $\forall i \in I$, $z_j^c\leftarrow z_j$, $\forall j \in J$;
\FOR {$i \in I$}
\STATE $w_i^p\leftarrow (f_i+\sum_{j\in J}c_{ij})/b_i$;
\ENDFOR
\STATE Rank all plants according to their priority indexes $w_i^p,\forall i \in I$;
\STATE $i=1$;
\WHILE {$\sum_{i\in I}b_iy_i^c<\sum_{k\in K}q_k$}
\STATE Select next plant $i'$ with the minimum priority index $w^p_{i'}$ from the set of remaining plants;
\STATE $y_{i'}^c\leftarrow 1$, $i\leftarrow i+1$;
\ENDWHILE
\STATE $i\leftarrow |I|$;    \%Improve
 the corrected component  $y^c=(y_{i'}^c), i'\in I$
\WHILE {$\sum_{i \in I}b_iy_i^c>\sum_{k \in K}q_k$}
\STATE Select a plant $i'$ with the maximum priority index $w^p_{i'}$;
\STATE $y_{i'}^c\leftarrow 0$;
\IF {$\sum_{i\in I}b_iy_i^c<\sum_{k\in K}q_k$}
\STATE $y_{i'}^c\leftarrow 1$;  
\STATE \textbf{break}
\ENDIF
\STATE $i\leftarrow i-1$;
\ENDWHILE
\FOR {$j \in J$}
\STATE $w_j^d\leftarrow (\sum_{i\in I}c_{ij}y_j+g_j+\sum_{k\in K}d_{jk})/p_j$;
\ENDFOR
\STATE Rank all depots according to their priority indexes $w_j^d,\forall j \in J$;
\STATE $j=1$;
\WHILE {$\sum_{j\in J}p_jz_j^c<\sum_{k\in K}q_k$}
\STATE Select a depot $j'$ with the minimum priority index $w^d_{j'}$;
\STATE $z_{j'}^c\leftarrow 1$, $j\leftarrow j+1$;
\ENDWHILE
\STATE $j\leftarrow |J|$;     \%Improve the corrected component  $z_{j'}^c$, $\forall j'\in J$
\WHILE {$\sum_{j \in J}p_jz_j^c>\sum_{k \in K}q_k$}
\STATE Select a depot $j'$ with the maximum priority index $w^d_{j'}$;
\STATE $z_{j'}^c\leftarrow 0$;
\IF {$\sum_{j \in J}p_jz_j^c<\sum_{k \in K}q_k$}
\STATE $z_{j'}^c\leftarrow 1$;
\STATE \textbf{break}
\ENDIF
\STATE $j\leftarrow j-1$;
\ENDWHILE
\STATE \textbf{Output}: The improved feasible individual $X^c=[y^c \quad  z^c]$.
\end{algorithmic}
\end{algorithm}

Once the initial population of the HEA/FA algorithm is produced, their fitness values are calculated exactly. In this paper, the objective function value is  directly   regarded as the fitness value.
That is to say, the smaller the fitness value of one individual, the better its quality.
In order to produce a sufficient number of solutions for the training set of the  ELM, the size of of the training set  is chosen to be $2N_p$. Then the first $N_p$ solutions with   better fitness values are selected to constitute the initial population fed to the GA. Subsequently, this population is updated through  genetic search operations described below.

\subsection{Selection operation}
Selection is invoked two times  in the optimization process of the HEA/FA.
\emph{Selection for reproduction} (step 7 in Algorithm 1) is carried out before  genetic operations are performed. This is performed on a purely random strategy without bias  to filter two individuals from the population. Moreover, \emph{selection for the individuals of next generation} (step 19 in Algorithm 1)  is performed after the offspring individuals of new generation have been produced.
In order to maintain the diversity of the population, the best $N_p$ solutions are chosen from the  combination of parents and offspring

\subsection{ Crossover operation}
In our algorithm, the CX reproduction operator is adopted to finish the crossover operation. The CX reproduction operator proposed for the quadratic assignment problem  (\citet*{Merz_MA_CX1999,Merz_MA_CX2000}),  preserves the information contained in two parents in the sense that all alleles of the offspring  are taken either from the first or from the second parent.
The operator does not carry out any implicit mutation operation since an element that is assigned to position $\imath$ in the offspring solution is also assigned to position $\imath$ in one or both parent solutions.

However, the CX operator was designed for the solution with permutation sequence, and is difficult to tackle the binary solutions in our problem. The CX operator is modified to achieve the crossover operation for the binary encoding scheme.
The procedure of the CX operator can be delineated as follows. Firstly, all elements found at the same positions in the two parents are assigned to the corresponding positions in the offspring.
Then, a uniformly distributed random number between zero and one is generated. The random number is then compared with 0.5. If the random number is less than  0.5, the element in the selected  position of the offspring is set to equal to the element of the same position in the first parent; otherwise, its corresponding position is occupied by the element from the second parent.  Subsequently, the next no-assigned position  is processed in the same way until all positions have been occupied.  Figure \ref{fig:fig_CXcrossover} illustrates the manner in which the CX reproduction operator performs.
\begin{figure}[htp]
  \centering
  \includegraphics[width=2.8in]{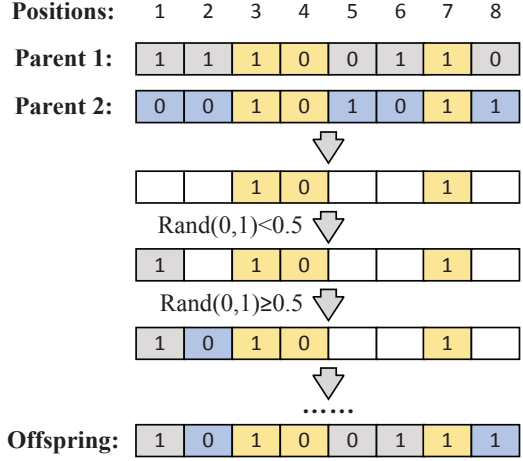}\\
  \caption{The CX reproduction operator}\label{fig:fig_CXcrossover}
\end{figure}

{In addition, the crossover probability $\rho_c$ is selected adaptively as in \citep*{Srinivas1994GA}.
Based on their strategy, the crossover probability and the mutation probability are increased when the population tends to get stuck at a local optimum and are decreased when the population is scattered in the search space of the algorithm. Meanwhile, if the fitness value of the selected individual is better than average fitness value of the population, the corresponding probability should be smaller to preserve the individual in next population; otherwise, the corresponding probability should be larger to generate new individuals in next population.}
  Let $f_{best}$ be the best fitness of the current population, $\bar{f}$ be the average fitness value of the population and $f'$ be the smaller of the fitness of the two individuals to be crossed. Then the crossover probability $\rho_c$  is calculated as follows
  \begin{equation}\label{eq:eq_pc}
  \rho_c=\left\{
\begin{aligned}
\rho_c^{min}+\left(\frac{f_{best}-f'}{f_{best}-\bar{f}}\right)\times(\rho_c^{max}-\rho_c^{min}), \quad f'<\bar{f}\\
\rho_c^{max},\quad f'\geq \bar{f}
\end{aligned}
\right.
\end{equation}
%
where $\rho_c^{max}$ is the maximum value of the crossover probability and $\rho_c^{min}$ is the minimum value of the crossover probability. When $\rho_c^{max}=0.9$ and $\rho_c^{min}=0.5$, the HEA/FA algorithm could provide the best performance based on  preliminary experiments.
\subsection{ Mutation operation}
Generally, the inversion of the mutant gene is commonly used in the GA for the binary optimization problems. Nevertheless, the offspring individual delivered by the inversion operation is probably infeasible due to the violation of the capacity constraints for the problem under study. For instance, assume an incumbent individual  just meets the requirement of the customers. When one plant is closed under the inversion operation, the production capacity of the remaining plants may be less than the demand of all customers. In this case, the candidate individual obtained by mutation becomes infeasible due to the violation of the capacity constraint.
To avoid the above mentioned problem of common mutation strategy, alternatively, the mutation operator used in the HEA/FA exchanges two randomly selected elements in a chosen individual.

The two components of the individual are handled by the exchange operator separately.
For the first $|I|$ binary values, two elements $\imath$ and $\jmath$ are randomly selected from the current individual and are swapped directly. The remaining $|J|$ binary values are operated in the same way. To describe the operation of the mutation operator, a simple example is shown in Figure \ref{fig:fig_mutation}.

\begin{figure}[htp]
  \centering
  \includegraphics[width=4in]{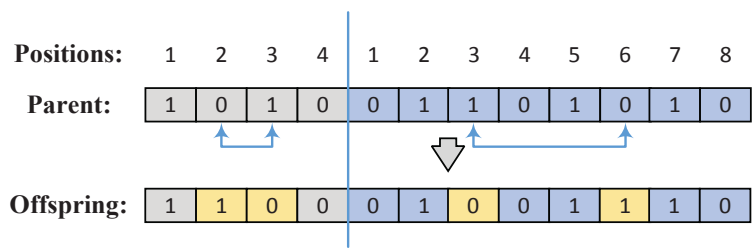}\\
  \caption{The Mutation operator
  }\label{fig:fig_mutation}
\end{figure}

The mutation probability adopts  the same adaptively updating strategy like equation \eqref{eq:eq_pc}. The expression for mutation probability $\rho_m$ is given as
\begin{equation}\label{eq:eq_pm}
  \rho_m=\left\{
\begin{aligned}
\rho_m^{min}+\left(\frac{f_{best}-f'}{f_{best}-\bar{f}}\right)\times(\rho_m^{max}-\rho_m^{min}), \quad f<\bar{f}\\
\rho_m^{max},\quad f\geq \bar{f}
\end{aligned}
\right.
\end{equation}
where $\rho_m^{min}$ and $\rho_m^{max}$ are the minimum and the maximum of the mutation probability, $f_{max}$ and $\bar{f}$ are the same as defined above, and $f $ is the fitness of the individual under mutation.
Based on the preliminary experiments, when $\rho_m^{max}$ and $\rho_m^{min}$ are set to 0.2 and 0.01 respectively, the algorithm could provide the best performance for the problem under study.
From the expression of $\rho_c$ and $\rho_m$, it is seen that $\rho_c$ and $\rho_m$ will get lower values for better individuals and get higher values for worse individuals.
\subsection{Local search strategy}
Generally, integrating a local search (LS) technique within a meta-heuristic will typically generate more competitive results. The main benefit of hybridizing evolutionary algorithms with a LS algorithm is to accelerate  the speed of convergence.
In our algorithm, the best individual of the population  in each iteration is improved by a local search.
The crucial idea of our LS strategy is generating  all possible  alternatives of the incumbent solution by inverting each element of  a solution, while each of these alternatives is corrected by the MIH algorithm in order to meet the feasibility.   If a candidate is better than the incumbent one in terms of the fitness value, the candidate individual will be accepted as a better solution.
The LS procedure is outlined in Algorithm \ref{alg:alg_LS}.
\begin{algorithm}[htp]
\begin{algorithmic}[1]\small
\caption{\label{alg:alg_LS} The Local Search Procedure}
\STATE Input: One individual $X$ and its fitness value;
\STATE The candidate individuals set $Q^*\leftarrow [\quad ]$;
\STATE $\kappa$=0;
\FOR {each element $\imath$  in the  individual $X$}
\STATE Perform inversion on the element $\imath$ to produce a new individual $X'$, while keeping other elements fixed;
\STATE Use the MIH algorithm to correct and refine the new individual $X'$;
\IF {the individual $X'$ is different from the incumbent one $X$}
\STATE Add the individual $X'$ to the set $Q^*$;
\STATE $\kappa$=$\kappa+1$;
\ENDIF
\ENDFOR
\STATE Estimate the individuals of the set $Q^*$ with the ELM fitness approximation;
\STATE Choose the best one from the set $Q^*$. If the best one is better than $X$, then update the individual $X$. Otherwise, $X$ keeps unchanged;
\STATE Output: the  individual $X$ and its fitness value.
\end{algorithmic}
\end{algorithm}

\subsection{Restarting strategy}
In order to maintain the diversity  of the population, it is common to use a restarting strategy in evolutionary algorithms \citep*{Ruiz2006461}. In the proposed HEA/FA algorithm,  a certain percentage of the population is substituted by randomly generated individuals with correction by the MIH algorithm. {Whenever the number of identical binary elements   in  the best individual and in the worst individual is more than  or equal to 0.9($|I|+|J|$),}
the restarting operation is performed on the population. The population is sorted in non-decreasing
 order of objective function values. Then the first 90\% of individuals from the sorted list are retained in the population. The remaining 10\% of individuals are replaced with newly generated individuals at random with correction by the MIH algorithm. 

\subsection{Fitness evaluation}
To avoid prematurely converging to a false optimum, the approximation model needs to be used together with the exact fitness function. For the individuals delivered by the genetic operation, their fitness values are approximated by the ELM-model. Since the best individual ($X_{\mathrm{best}}$) found so far   plays a central role in ensuring the whole population to converge to a true optimum, the HEA/FA always calculates the fitness of the best individual using the exact fitness function to ensure a correct convergence.
Moreover, the set of the elite individuals with better fitness values are selected from the current population. In this paper, the proportion of the elites is set to $10\%$.
That is to say, the number of  elite individuals $N_e= 10\%\times N_p$.

In each  generation of the HEA/FA, the true fitness values are calculated for the best individual and the elites in the population. These individuals in turn are added to the  the training set $S$ to retrain and refine the ELM. As the evolution of the training set, the deviation between the exact fitness value and the approximated fitness value is expected to decrease, so that the evolutionary algorithm evolves better estimations produced by the ELM model.

\section{Numerical results}\label{sec:secNR}
In this section, the numerical experiments  are performed for evaluating the performance of the proposed algorithm.
All the reported computational experiments presented hereunder were conducted on a  personal computer with an Intel i7 3.6 GHz  processor and 8GB of RAM. The performance of each algorithm is measured by the relative percentage deviation (RPD) defined by the equation
\begin{equation}\label{eq:eq_rpd}
  RPD(\%)=(Z(\mathrm{alg})-Z(\mathrm{LB}))\times100/Z(\mathrm{LB})
\end{equation}
where $Z(\mathrm{alg})$ is the solution value delivered by a  specific algorithm and $Z(\mathrm{LB})$ is the lower bound of the corresponding instance. $Z(\mathrm{LB})$ can be obtained by solving the LP model with the continuous decision variables $y_i $ and $z_j$, $i \in I, j \in J$.

For assessing the performance of the HEA/FA algorithm, a well-defined set of instances developed by \citet*{Fernandes2014200} were used in this paper.
The benchmark instances consist of five classes of instances by varying seven parameters of the considered problem: $b_i,f_i,c_{ij},p_j,g_j,d_{jk}, q_k$. The value intervals of these parameters are outlined in Table \ref{tab:tab_bdata}. For each instance, the values of the seven parameters are randomly generated from an uniform random distribution. In \citep*{Fernandes2014200}, the number of plants was set to 50 and 100, respectively.
In all instances, the number of depots is twice as much as the number of  plants,    and the number of  customers is twice as much as that of depots, so that $|K|=2|J|=4|I|$. For simplicity, the number of plants is used to indicate the size of the considered instances like in \citep*{Fernandes2014200}.
 \begin{table}[htbp]\scriptsize
 \centering
  \caption{Value intervals of the five classes of instances introduced by \citet*{Fernandes2014200}}
    \begin{tabular}{rrrrrr}
    \toprule
    Parameter & Class 1 & Class 2 & Class 3 & Class 4 & Class 5 \\
    \midrule
    $b_i$    & [2B 5B] & [5B 10B] & [15B 25B] & [5B 10B] & [5B 10B] \\
    $f_i$    & $[2\times10^4$  $3\times10^4]$ & $[2\times10^4$  $3\times10^4]$ & $[2\times10^4$  $3\times10^4]$ & $[2\times10^4$  $3\times10^4]$ & $[2\times10^4$  $3\times10^4]$ \\
    $c_{ij}$   & [35 45] & [35 45] & [35 45] & [50 100] & [35 45] \\
    $p_j$    & [2P 5P] & [5P 10P] & [15P 25P] & [5P 10P] & [5P 10P] \\
    $g_j$    &$[8\times10^3$  $1.2\times10^4]$ & $[8\times10^3$  $1.2\times10^4]$ & $[8\times10^3$  $1.2\times10^4]$ & $[8\times10^3$  $1.2\times10^4]$ & $[8\times10^3$  $1.2\times10^4]$ \\
    $d_{jk}$   & [55 65] & [55 65] & [800 1000] & [50 100] & [800 1000] \\
    $q_k$    & [10 20] & [10 20] & [10 20] & [10 20] & [10 20] \\
    \bottomrule

    \multicolumn{6}{c}{\leftline{ Note: $B=\sum_{k \in K}q_k/|I|$, $P=\sum_{k \in K}q_k/|J|$.}}
    \end{tabular}%
  \label{tab:tab_bdata}%
\end{table}%

The proposed HEA/FA approach was coded in Matlab 7.14 under a Windows 7 environment. The parameters selection can significantly affect the quality of a algorithm, in terms of solution performance and computational time. In our algorithm, the following three parameters are necessary to be considered: $N_p$, $t_{\mathrm{max}}$ and $t_{\mathrm{nip}}$.
  Based on the preliminary tests, when $t_{\mathrm{max}}=200$ and $t_{\mathrm{nip}}=50$, the HEA/FA could provide good performance within a reasonable computational time.  For the population size, the following candidate values are tuned for our algorithm $N_{p}$=40, 50, 60, 70 and 80. A set of  test instances with 50 plants were randomly generated in the same way as in \citep*{Fernandes2014200}. These test instances were solved by these HEA/FAs with different population sizes.
  The computational results were examined  by means of  the one-way Analysis of Variance (ANOVA) test. The means plot and 99\% confidence level Tukey's Honestly Significant Difference (HSD) intervals for these 5 candidate values of $N_p$ are described in Figure \ref{fig:fig_pnum}.
  As can be seen from the figure, the Tukey's HSD intervals for different $N_p$ values have overlapping.  which indicates there is no significant difference among the five $N_p$ values. However, when the population size is 60, the mean results delivered by the HEA/FA are better compared with the other values. Therefore, the parameter $N_p$ is set to 60 in the following computations.
\begin{figure}
  \centering
  \includegraphics[width=3.5 in]{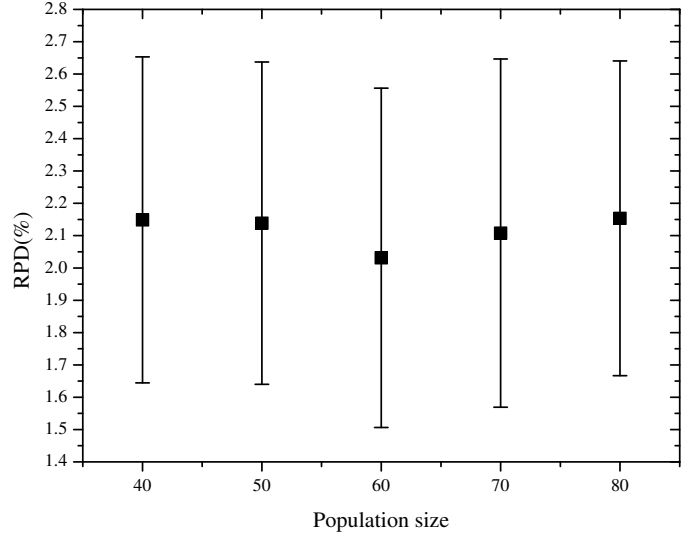}\\
  \caption{Means plot and 99\% confidence level Tukey's HSD intervals for population size $N_p$}\label{fig:fig_pnum}
\end{figure}

Due to the stochastic nature of the HEA/FA approach, it is run five times for a given instances to reach  reliable results. The best objective function value among 5 runs and the average objective value are recorded for our algorithm. In addition, our results are compared with the state-of-the-art GA proposed by \citet*{Fernandes2014200} for the TSCFLP.

Table \ref{tab:tab_r50} reports the computational results of the instances with 50 plants. In the table, when the objective value of HEA/FA is better than that of GA, the corresponding value is in boldface font.
It can be observed that the HEA/FA gives good solutions in a reasonable computational time, and the average RPD values delivered by the HEA/FA are slightly larger than that of the GA. Remarkably, the HEA/FA gives better solutions for six instances compared with the GA.

Figure \ref{fig:fig_r50} shows the performances of the two algorithms under different problem classes.  In the figure, the RPD value is the mean of five problem instances for each class.
It is clearly seen from Figure \ref{fig:fig_r50} that the two algorithms show similar performance in solving the instances of Class 1, Class2 and Class 5. The average RPD values obtained by the HEA/FA are larger than that of GA in the instances of Class 3 and Class 4.

\begin{table}[htbp]
\tiny
  \centering
  \caption{Computational results obtained by GA and HEA/FA for instances with 50 plants}
    \begin{tabular}{rrrrrrrrrr}
    \toprule
    Class & Instance & LB    &  \multicolumn{2}{c}{ GA}    & \multicolumn{5}{c}{ HEA/FA}  \\
      \cmidrule[0.05em](r){4-5}
    \cmidrule[0.05em](lr){6-10}

          &       &       & Z     & RPD & $Z_{\mathrm{min}}$  & $Z_{\mathrm{average}}$ & ${\mathrm{RPD}}_{\mathrm{min}}$ & ${\mathrm{RPD}}_{\mathrm{average}}$ & Time(s) \\
           \midrule
    1     & 1     & 721209.6 & 722178.0 & 0.13  & 722178 & 722946.0 & 0.13  & 0.24  & 416.17 \\
          & 2     & 730451.6 & 733350.4 & 0.40  & \textbf{732737} & \textbf{733065.4} & 0.31  & 0.36  & 492.78 \\
          & 3     & 731885.3 & 733664.2 & 0.24  & \textbf{733473} & \textbf{733650.4} & 0.22  & 0.24  & 316.08 \\
          & 4     & 721515.0 & 727325.0 & 0.81  & \textbf{725387} & \textbf{725741.0} & 0.54  & 0.59  & 455.55 \\
          & 5     & 713633.8 & 719512.6 & 0.82  & 719771 & 720559.2 & 0.86  & 0.97  & 463.69 \\
                &       &       &       &       &       &       &       &       &  \\
    2     & 1     & 479860.2 & 492747.0 & 2.69  & 492987 & 493381.0 & 2.74  & 2.82  & 230.84 \\
          & 2     & 483072.2 & 494205.0 & 2.30  & 494366 & 494719.0 & 2.34  & 2.41  & 199.11 \\
          & 3     & 486018.5 & 496435.4 & 2.14  & \textbf{495089} & \textbf{495204.2} & 1.87  & 1.89  & 167.28 \\
          & 4     & 482374.6 & 492215.8 & 2.04  & 492349 & 492611.8 & 2.07  & 2.12  & 162.13 \\
          & 5     & 474803.3 & 489711.0 & 3.14  & \textbf{489625} & \textbf{489625.0} & 3.12  & 3.12  & 170.57 \\
                &       &       &       &       &       &       &       &       &  \\
    3     & 1     & 2608799.8 & 2688951.0 & 3.07  & 2690669 & 2694875.2 & 3.14  & 3.30  & 117.09 \\
          & 2     & 2616252.3 & 2697803.8 & 3.12  & 2702580 & 2704217.2 & 3.30  & 3.36  & 116.67 \\
          & 3     & 2598276.5 & 2679038.0 & 3.11  & 2684051 & 2686436.2 & 3.30  & 3.39  & 161.80 \\
          & 4     & 2612533.8 & 2692662.0 & 3.07  & 2695644 & 2696585.2 & 3.18  & 3.22  & 148.86 \\
          & 5     & 2568855.9 & 2646182.0 & 3.01  & 2650230 & 2650844.2 & 3.17  & 3.19  & 158.15 \\
    4     & 1     & 525294.1 & 541803.0 & 3.14  & 542944 & 543897.2 & 3.36  & 3.54  & 216.81 \\
          & 2     & 526911.7 & 539178.0 & 2.33  & 541357 & 542271.2 & 2.74  & 2.92  & 199.39 \\
          & 3     & 532592.3 & 546738.4 & 2.66  & \textbf{546365} & \textbf{546534.8} & 2.59  & 2.62  & 247.29 \\
          & 4     & 529372.0 & 542750.0 & 2.53  & 543287 & 544224.6 & 2.63  & 2.81  & 261.75 \\
          & 5     & 521470.1 & 537806.4 & 3.13  & 538047 & 538047.8 & 3.18  & 3.18  & 182.24 \\
                &       &       &       &       &       &       &       &       &  \\
    5     & 1     & 2743547.2 & 2776346.8 & 1.20  & \textbf{2775499} & 2776499.2 & 1.16  & 1.20  & 199.23 \\
          & 2     & 2752021.4 & 2781496.0 & 1.07  & 2782534 & 2783844.0 & 1.11  & 1.16  & 259.72 \\
          & 3     & 2737769.1 & 2767842.2 & 1.10  & 2771044 & 2772909.0 & 1.22  & 1.28  & 202.71 \\
          & 4     & 2748216.2 & 2777619.0 & 1.07  & 2778388 & 2779207.0 & 1.10  & 1.13  & 203.32 \\
          & 5     & 2702350.0 & 2736077.8 & 1.25  & 2737843 & 2738494.6 & 1.31  & 1.34  & 182.20 \\
    \multicolumn{2}{c}{Average} &       &       & 1.98  &       &       & 2.03  & 2.10  & 237.26 \\
    \bottomrule
    \end{tabular}%
  \label{tab:tab_r50}%
\end{table}%

\begin{figure}
  \centering
  \includegraphics[width=3.5 in]{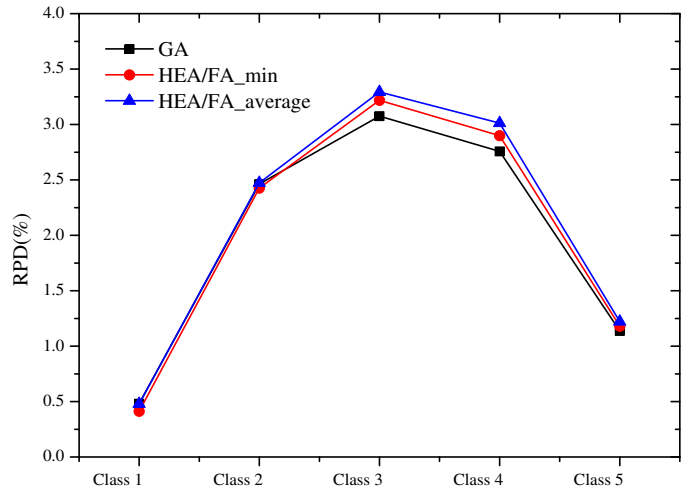}\\
  \caption{Average RPD values delivered by GA and HEA/FA for instances with 50 plants }\label{fig:fig_r50}
\end{figure}

The results of problem instances with 100 plants obtained by the two algorithms are listed in Table \ref{tab:tab_r100}.
In the table, the average RPD value of the best solution among five runs given by the HEA/FA is 0.85\%, which is less than the average RPD value of 0.96\%  given by the GA. Moreover, the average RPD of the mean objective function value delivered by the HEA/FA is 0.94\%,which is slightly less than the one of GA. There are 12 instances in which the quality of the average solutions obtained by HEA/FA is better than that of GA. And the best solutions yielded by the HEA/FA are better than that of the GA in 14 out of 25 instances. Note that these objective  function values are marked by the bold-faced font.
In the same way, the average RPD values given by the two algorithms under different problem classes are plotted in Figure \ref{fig:fig_r100}. The HEA/FA significantly outperforms the GA in the problem instances with Class 1, Class 2 and Class 5. But the proposed HEA/FA is   not as good at handling the problem instances with Class 3 and Class 4.

Furthermore, the computational time of the HEA/FA is listed in Table \ref{fig:fig_r50} and \ref{fig:fig_r100}. The longest time of HEA/FA is 492.78 sec  and 1514.37 sec in solving problem instances with 50 plants and 100 plants, respectively. Most computational time is still consumed by exact fitness evaluation in the search process.
And the computational time of the HEA/FA is acceptable for obtaining a good quality solution to the underlying problem.
Based on the above analysis, the HEA/FA could be regarded as an alternative algorithm in solving the TSCFLP, especially for  large-sized problems.

\begin{table}[htbp]
\tiny
  \centering
  \caption{Computational results obtained by GA and HEA/FA for instances with 100 plants}
    \begin{tabular}{rrrrrrrrrr}
    \toprule
    Class & Instance & LB    &  \multicolumn{2}{c}{ GA}    & \multicolumn{5}{c}{ HEA/FA}  \\
      \cmidrule[0.05em](r){4-5}
    \cmidrule[0.05em](lr){6-10}

          &       &       & Z     & RPD & $Z_{\mathrm{min}}$  & $Z_{\mathrm{average}}$ & ${\mathrm{RPD}}_{\mathrm{min}}$ & ${\mathrm{RPD}}_{\mathrm{average}}$ & Time(s) \\
           \midrule
    1     & 1     & 1475951.6 & 1484057.40 & 0.55  & \textbf{1480295} & \textbf{1480769} & 0.29  & 0.33  & 1341.67 \\
          & 2     & 1462736.1 & 1477503.40 & 1.01  & \textbf{1466719} & \textbf{1469070} & 0.27  & 0.43  & 1514.37 \\
          & 3     & 1492162.6 & 1497213.00 & 0.34  & \textbf{1495634} & 1499353 & 0.23  & 0.48  & 1481.83 \\
          & 4     & 1459076.4 & 1466182.20 & 0.49  & \textbf{1462752} & \textbf{1463128} & 0.25  & 0.28  & 1122.33 \\
          & 5     & 1490741.8 & 1500689.20 & 0.67  & \textbf{1493341} & \textbf{1494275} & 0.17  & 0.24  & 1501.11 \\
           &       &       &       &       &       &       &       &       &  \\
    2     & 1     & 970908.5 & 979526.60 & 0.89  & \textbf{977047} & \textbf{977685.2} & 0.63  & 0.70  & 944.83 \\
          & 2     & 965908.5 & 973034.40 & 0.74  & \textbf{970486} & \textbf{971307.6} & 0.47  & 0.56  & 908.30 \\
          & 3     & 975499.7 & 989362.00 & 1.42  & \textbf{977461} & \textbf{977912} & 0.20  & 0.25  & 922.00 \\
          & 4     & 973019.1 & 978502.00 & 0.56  & \textbf{978484} & 978727 & 0.56  & 0.59  & 1083.61 \\
          & 5     & 941567.0 & 952067.80 & 1.12  & \textbf{949666} & \textbf{950594.8} & 0.86  & 0.96  & 1159.33 \\
           &       &       &       &       &       &       &       &       &  \\
    3     & 1     & 5213566.2 & 5298518.40 & 1.63  & 5313347 & 5319194 & 1.91  & 2.03  & 855.68 \\
          & 2     & 5191320.9 & 5278225.40 & 1.67  & 5293152 & 5294478 & 1.96  & 1.99  & 1220.48 \\
          & 3     & 5145991.1 & 5227517.00 & 1.58  & 5237805 & 5240865 & 1.78  & 1.84  & 1193.12 \\
          & 4     & 5225601.2 & 5316646.00 & 1.74  & 5328880 & 5330736 & 1.98  & 2.01  & 968.64 \\
          & 5     & 5163182.1 & 5251934.40 & 1.72  & 5265931 & 5270409 & 1.99  & 2.08  & 942.29 \\
    4     & 1     & 1052171.8 & 1060799.00 & 0.82  & 1061704 & 1063117 & 0.91  & 1.04  & 878.64 \\
     &       &       &       &       &       &       &       &       &  \\
          & 2     & 1043552.6 & 1053288.60 & 0.93  & \textbf{1052133} & \textbf{1052876} & 0.82  & 0.89  & 672.60 \\
          & 3     & 1050682.6 & 1070421.00 & 1.88  & \textbf{1061159} & \textbf{1064202} & 1.00  & 1.29  & 877.26 \\
          & 4     & 1044570.7 & 1054638.00 & 0.96  & 1058764 & 1061535 & 1.36  & 1.62  & 797.26 \\
          & 5     & 1053868.7 & 1060621.00 & 0.64  & 1063793 & 1065005 & 0.94  & 1.06  & 843.02 \\
           &       &       &       &       &       &       &       &       &  \\
    5     & 1     & 5486098.3 & 5512662.00 & 0.48  & 5519416 & 5520099 & 0.61  & 0.62  & 967.57 \\
          & 2     & 5461680.0 & 5487604.20 & 0.47  & \textbf{5484270} & \textbf{5485724} & 0.41  & 0.44  & 1017.10 \\
          & 3     & 5425391.4 & 5458935.80 & 0.62  & 5460046 & 5464792 & 0.64  & 0.73  & 1371.96 \\
          & 4     & 5494811.0 & 5523513.80 & 0.52  & 5526703 & 5526910 & 0.58  & 0.58  & 1084.43 \\
          & 5     & 5442621.0 & 5468022.80 & 0.47  & \textbf{5465944} & \textbf{5467966} & 0.43  & 0.47  & 1162.63 \\
    \multicolumn{2}{c}{Average} &       &       & 0.96  &       &       & 0.85  & 0.94  & 1073.28 \\
    \bottomrule
    \end{tabular}%
  \label{tab:tab_r100}%
\end{table}%

\begin{figure}
  \centering
  \includegraphics[width=3.5 in]{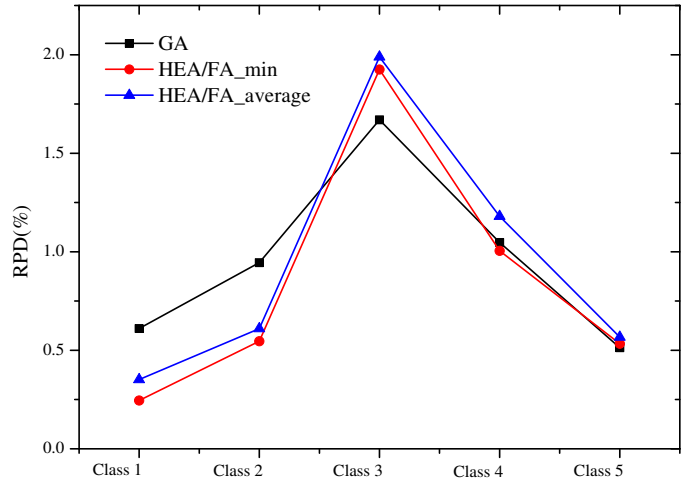}\\
  \caption{
  Average RPD values delivered by GA and HEA/FA for instances with 100 plants }\label{fig:fig_r100}
\end{figure}

\section{Conclusions}\label{sec:secCon}
In this paper, a hybrid evolutionary algorithm frame with fitness approximation  was proposed to solve TSCFLP.  To improve the performance of the algorithm, the following tricks are adopted:  two heuristics are introduced during the population initialization, and a local search  based on inversion operation is integrated to refine the best individual at each iteration. In addition, probabilistically adaptive crossover and mutation are used to prevent the algorithm from getting stuck at a local optimal solution. In order to reduce the  time consumption in fitness evaluation operation, the fitness values of most individuals are calculated by the surrogate model based on extreme learning machine. However, the elite solutions are assessed by the exact fitness evaluation for getting a proper balance between accuracy and time efficiency of fitness evaluation. The performance of the proposed algorithm was tested and evaluated on two sets of benchmark instances.  The computational results show that the HEA/FA in general delivers good results for the TSCFLP in a reasonable time, especially in large-sized instances. Compared with the state-of-the-art genetic algorithm, the HEA/FA provides better solutions  for   large-sized problem.

Further research includes the consideration of additional characteristics in the problem setting to make it even more realistic. For instance, some loading/discharging operations such as  handling costs studied by \citet*{Li2014957} in   facilities might prove useful. Additionally, traffic situations such as asymmetrical transport distance matrices could be considered in more complex scenarios.
For the HEA/FA, the probability of applying it to solve other combinatorial optimization problems could be developed, such as production scheduling problems and facility layout problems.
\section*{Acknowledgment}
This work is supported by the National Natural Science
Foundation of China (No. 51405403) and
China Postdoctoral Science Foundation funded project (No. 2015M582566).



\end{document}